\documentclass[11pt, a4paper]{article}

\usepackage[pdftex]{graphicx}%
\usepackage{subcaption}
\usepackage{multirow}%
\usepackage{amsmath,amssymb,amsfonts}%
\usepackage{amsthm}%
\usepackage{mathrsfs}%
\usepackage{mathtools}
\usepackage{multicol}

\newtheorem{Th}{Theorem}[section]

\newtheorem{Prop}[Th]{Proposition}

\newtheorem{Prob}[Th]{Problem}
\newtheorem{Rem}[Th]{Remark}

\newtheorem{Conj}[Th]{Conjecture}

\newcommand{\N}{\mathbb{N}}
\newcommand{\R}{\mathbb{R}}
\newcommand{\dt}{\Delta t}
\newcommand{\vmax}{v_{\rm max}}
\newcommand{\DS}{\displaystyle}

\title{Threshold dynamics in time-delay systems:\\ 
polynomial $\beta$-control in a pressing process\\ 
and connections to blow-up} 

\author{
Masato Kimura\thanks{Faculty of Mathematics and Physics, Kanazawa University, Kakuma, Kanazawa, Ishikawa 920-1192, Japan (mkimura@se.kanazawa-u.ac.jp)}
\and 
Hirotaka Kuma\thanks{Towa Seiki Co. Ltd., 5-26-7 Toei, Anjo, Aichi 446-0007, Japan (hkuma@tsco.co.jp)}
\and 
Yikan Liu\thanks{Department of Mathematics, Kyoto University, Kitashirakawa-Oiwake, Sakyo, Kyoto 606-8502, Japan (liu.yikan.8z@kyoto-u.ac.jp)}
\and 
Kazunori Matsui\thanks{Department of Logistics and Information Engineering, Tokyo University of Marine Science and Technology, 2-1-6 Etchujima, Koto, Tokyo 135-8533, Japan (kmat002@kaiyodai.ac.jp)}
\and 
Masahiro Yamamoto\thanks{Department of Mathematical Sciences, the University of Tokyo, 3-8-1 Komaba, Meguro, Tokyo 153-8914, Japan (myama@ms.u-tokyo.ac.jp)}
\and 
Zhenxing Yang\thanks{Division of Mathematical and Physical Sciences, Kanazawa University, Kakuma, Kanazawa, Ishikawa 920-1192, Japan}
}
\date{}

\begin{document}
\maketitle

\begin{abstract}
This paper addresses a press control problem in straightening machines with small time delays due to system communication. To handle this, we propose a generalized \emph{$\beta$-control} method, which replaces conventional linear velocity control with a polynomial of degree $\beta \ge 1$. The resulting model is a delay differential equation (DDE), for which we derive basic properties through nondimensionalization and analysis.
Numerical experiments suggest the existence of a threshold initial velocity separating overshoot and non-overshoot dynamics, which we formulate as a conjecture. Based on this, we design a control algorithm under velocity constraints and confirm its effectiveness. We also highlight a connection between threshold behavior and finite-time blow-up in DDEs.
This study provides a practical control strategy and contributes new insights into threshold dynamics and blow-up phenomena in delay systems.
\end{abstract}

{\small

}

\maketitle
\begin{multicols}{2}
\section{Introduction}\label{sec:Introduction}
\setcounter{equation}{0}

Optimal control is essential for modern automated manufacturing, where increasingly sophisticated technologies demand maximum efficiency. Although control theory has a long history, it continues to evolve in response to technological advancements. The mathematical modeling and analysis of the target system play a critical role in developing effective control algorithms. When the process at the core of the system is described by a well-structured mathematical model, it becomes possible to propose robust strategies for optimal control.

~\\
One of the most widely used methods is PID control \cite{AH95,O'Dwyer09}, which remains prevalent in industrial applications. However, time delays—arising from communication and computation in feedback loops—can destabilize such systems, leading to overshoot, instability, or even chaotic behavior.

Delay differential equations (DDEs) \cite{Driver77,LL69} provide a mathematical framework for modeling systems with time delays and have been extensively studied, especially in engineering contexts \cite{HVL93, KM92}. Nonetheless, the dynamics of DDEs are often difficult to analyze, even for simple cases. Challenges such as instability, chaos \cite{IM87,MG77,WSG19}, and finite-time blow-up \cite{EJ06,EIIN21,IN22} remain active areas of research.

In this study, we address a press control problem in a straightening machine, modeled by a specific class of DDEs, where the initial velocity serves as the control variable. We first derive a nondimensionalized DDE model and clarify the basic mathematical properties of its solutions. Numerical experiments suggest the existence of a critical threshold for the initial value that separates two types of behavior: overshooting and asymptotic convergence. We formulate this observation as a conjecture and, under its validity, propose an effective control algorithm. This method also accommodates a constraint on maximum velocity.

Moreover, the conjectured threshold phenomenon has another implication. It enables an interpretation of the system's behavior in terms of finite-time blow-up. We demonstrate a connection between the press control problem and blow-up dynamics, offering new insight into blow-up phenomena in DDEs.

The structure of the paper is as follows. Section~\ref{sec:2} introduces the industrial context and formulates the control problem based on empirical data, presenting a generalized $\beta$-control model. Section~\ref{sec:num} derives the nondimensionalized DDE, analyzes its mathematical properties, and proposes the central conjecture based on numerical evidence. Section~\ref{sec:application} applies the conjecture to construct a practical control algorithm under a maximum velocity constraint, validated through further simulations. Section~\ref{sec:blowup} reinterprets these results as finite-time blow-up phenomena for a related DDE, allowing us to analyze blow-up rates and threshold behavior. Finally, Section~\ref{sec:conclusion} summarizes our results and reflects on the broader significance of this work, illustrating how a practical industrial problem can lead to new mathematical developments in control and delay systems.

\section{Industrial Background and Problem Formulation}\label{sec:2}
\setcounter{equation}{0}

\subsection{Straightening press and process description}
We focus on a straightening press machine used to correct bends in heat-treated automotive shafts. These shafts often exhibit residual deformation after heat treatment, which can lead to undesirable vibration and noise during operation. Thus, precise straightening is necessary. The straightening process typically involves the following steps:

\begin{enumerate}
    \item Measure the shaft strain using several sensors placed along its length.
    \item Determine the pressing position, direction, and stroke length required to straighten the shaft.
    \item Control the press machine to perform pressing according to the parameters decided in step 2.
    \item Repeat steps 1--3 until the strain falls below the specified tolerance.
\end{enumerate}

For more details on shaft strain measurement, see \cite{Kawai16,LYK16,YKK16}.  
In this study, we focus on step 3: determining how to accurately press the shaft by a prescribed stroke. In particular, we consider the case where the press velocity 
$V$ is determined by the following generalized control rule:
\begin{align}\label{old_press}
    V(X) = v_0\left|\frac{\ell - X}{\ell}\right|^\beta,
\end{align}
where $\beta \ge 1$, $v_0 > 0$ is the initial velocity, 
$\ell$ is the target pressing stroke, and $X=X(t)$ is the position measured by the sensor
at time $t\ge 0$ with the initial position $X(0)=0$.

\begin{center}
  \includegraphics[width=\columnwidth,clip]{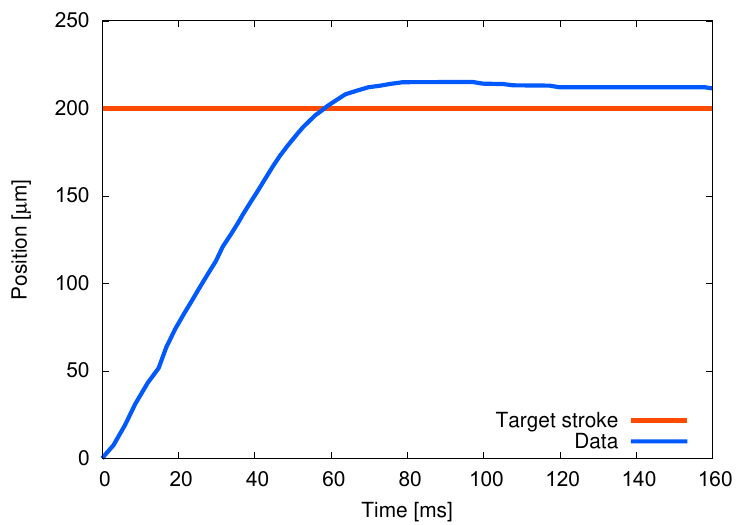}
  \captionof{figure}{The target pressing stroke (red line) and the actual sensor position (blue line), illustrating the overshoot phenomenon due to time delay. The press is stopped after overshooting.}
  \label{fig_overshoot}
\end{center}

\begin{center}
  \includegraphics[width=\columnwidth,clip]{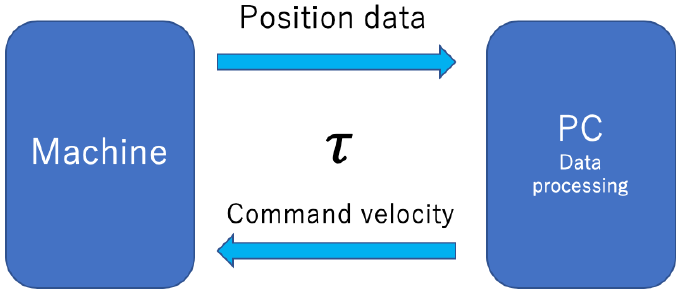}
  \captionof{figure}{Schematic of signal flow in the straightening press machine.}
  \label{fig_timedelay}
\end{center}

We refer to this control law as the \emph{$\beta$-control method} in this paper.
In conventional control systems, a simpler rule with exponent 
$\beta=1$ was often employed for ease of implementation. 
Our proposed method generalizes this by introducing the exponent
$\beta$ explicitly, allowing for more flexible and precise control of the press dynamics.

\subsection{Modeling time delay and overshoot phenomena}\label{subsec:2.2}

\begin{center}
  \begin{minipage}{\linewidth}\centering
    \includegraphics[width=\linewidth]{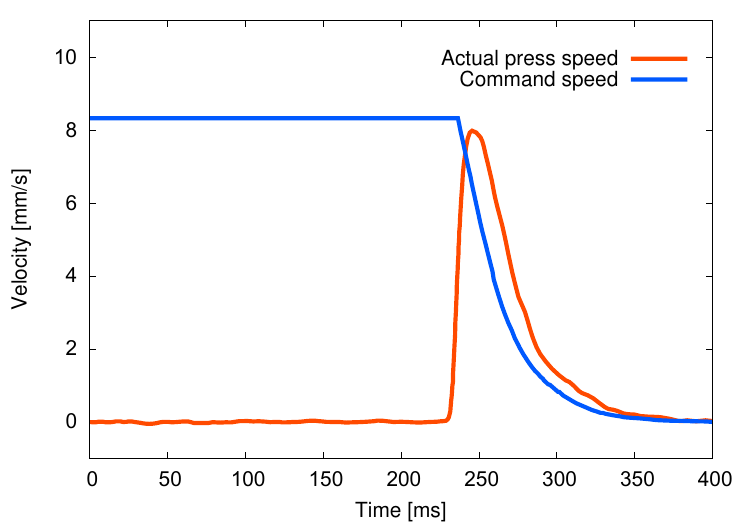}
    \par\small (a)
  \end{minipage}

  \vspace{0.5em}

  \begin{minipage}{\linewidth}\centering
    \includegraphics[width=\linewidth]{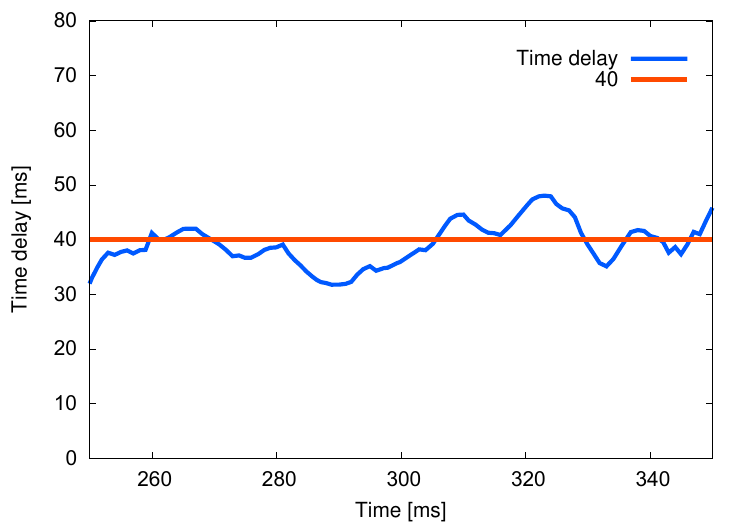}
    \par\small (b)
  \end{minipage}

  \captionof{figure}{Example of press machine data. (a) Comparison of the actual and commanded press velocities. The actual velocity is valid only after $t = 250$ ms. (b) Estimation of time delay $\tau$ as the horizontal gap between the two curves. The average delay is approximately $\tau = 40$ ms.}
  \label{fig_realdelay}
\end{center}

In the absence of delay, the differential equation $dX/dt = V(X)$ with $X(0)=0$ ensures that $X(t)$ converges to $\ell$ as $t \to \infty$. However, in practice, the system often overshoots the target position, as shown in Figure~\ref{fig_overshoot}. One major cause is the time delay between position measurement and actuation, introduced by signal transmission, data processing, and machine characteristics (Figure~\ref{fig_timedelay}). Figure~\ref{fig_overshoot} is based on experimental data from Towa Seiki Co. Ltd. using a conventional control algorithm.

Let $x(t)$ denote the true shaft position at time $t$. The position data used in the velocity command is then modeled as
\begin{align}\label{delay_pos}
    X(t) = x(t - \tau),
\end{align}
where $\tau > 0$ is the time delay.

The delay $\tau$ can be estimated by comparing the commanded velocity $V(X)$ with the measured actual press velocity $V_\text{meas}$ under controlled conditions. Figure~\ref{fig_realdelay} illustrates an example. From the data after $t = 250$ ms, the delay is estimated to be approximately $\tau = 40$ ms.

Hereafter, we assume that $\tau$ is a given positive constant independent of time and the velocity command.

We now formulate a delay differential equation (DDE) that captures this behavior. Given parameters $v_0, \ell, \tau > 0$ with $v_0 \tau < \ell$, and $\beta \ge 1$, the press velocity $dx/dt$ is assumed to be given by:
\begin{align}
\frac{dx}{dt}(t)
    &= \begin{cases}
       {\DS v_0 \left|\frac{\ell - x(t - \tau)}{\ell} \right|^\beta }& (t > \tau), \\
        v_0 & (0 \le t \le \tau),
    \end{cases}\label{vel_sys}\\
x(0) &= 0.\label{x_ini}
\end{align}
The condition $v_0 \tau < \ell$ ensures that the shaft does not overshoot the target stroke $\ell$ during the delay interval. Based on this model, we consider the following control problem.

\medskip
\noindent
\textbf{Press Control Problem:}  
Given parameters $\ell$, $\tau$, and $\beta$, find the optimal initial velocity $v_0$ that achieves the target stroke without overshoot and in the shortest time.

\section{Nondimensionalized DDE Model}\label{sec:num}
\setcounter{equation}{0}

\subsection{Scaling and basic properties}\label{subsec:scale}

The system \eqref{vel_sys} together with the initial condition \eqref{x_ini} is equivalent to the following problem.  
For \( v_0, \ell, \tau > 0 \) with \( v_0\tau < \ell \), and \( \beta \ge 1 \), find \( x:[0,\infty)\rightarrow\R \) such that
\begin{align}\label{sys_general}
\left\{
\begin{aligned}
    \frac{dx}{dt}(t)
    &= v_0 \left|\frac{\ell - x(t-\tau)}{\ell}\right|^\beta 
    &&(t > \tau), \\
    x(t) &= v_0 t && (0 \le t \le \tau).
\end{aligned}
\right.
\end{align}

To nondimensionalize system \eqref{sys_general}, we introduce the following dimensionless quantities:
\begin{align}\label{def_normalize}
s \coloneqq \frac{t}{\tau}, \quad
u(s) \coloneqq \frac{x(\tau s)}{\ell}, \quad
w_0 \coloneqq \frac{\tau v_0}{\ell} \in (0,1).
\end{align}
Then, for all \( s > 1 \), we have
\begin{align*}
    \frac{du}{ds}(s) 
    &= \frac{\tau}{\ell} \frac{dx}{dt}(\tau s) 
    = \frac{\tau v_0}{\ell} \left| \frac{\ell - x(\tau s - \tau)}{\ell} \right|^\beta \\
    &= w_0 \left|1 - u(s - 1)\right|^\beta,
\end{align*}
while for \( 0 \le s \le 1 \), it follows that
\[
    u(s) = \frac{x(\tau s)}{\ell} = \frac{v_0 \tau s}{\ell} = w_0 s.
\]
Hence, substituting \eqref{def_normalize} into \eqref{sys_general} gives the following nondimensionalized system:
\begin{align*}
\left\{
\begin{aligned}
    \frac{du}{ds}(s)
    &= w_0\left|1 - u(s - 1)\right|^\beta && (s > 1), \\
    u(s) &= w_0 s && (0 \le s \le 1),
\end{aligned}
\right.
\end{align*}
with \( 0 < w_0 < 1 \).

For convenience in the discussion that follows, we rename the independent variable \( s \) back to \( t \):
\begin{align}\label{sys_normalize}
\left\{
\begin{aligned}
    \frac{du}{dt}(t)
    &= w_0\left|1 - u(t - 1)\right|^\beta && (t > 1), \\
    u(t) &= w_0 t && (0 \le t \le 1).
\end{aligned}
\right.
\end{align}

Using standard analytical techniques for delay differential equations (e.g., \cite[Theorem 3.1]{Smith11}), we obtain the following theorem.

\begin{Th}
Suppose \( \beta > 0 \) and \( w_0 \in \R \). Then there exists a unique function \( u \in C^1([0,\infty)) \) that satisfies \eqref{sys_normalize}.
\end{Th}

\begin{proof}
Define \( f(u) \coloneqq w_0 |1 - u|^\beta \). Clearly, \( f \in C^0(\R) \).  
We construct \( u(t) \) inductively. For each \( k \in \N \), define
\[
u(t) = u(k) + \int_k^t f(u(s - 1))\,ds \quad \text{for } t \in [k, k+1],
\]
where \( u(t) = w_0 t \) for \( t \in [0,1] \).  
This defines a continuous function \( u \in C^0([0,\infty)) \) satisfying \( \frac{du}{dt}(t) = f(u(t - 1)) \) for \( t \in (1,\infty) \setminus \N \), and hence \( u \in C^1([1,\infty)) \).  

To show differentiability at \( t = 1 \), note that
\[
\frac{d^\pm u}{dt}(1) = \lim_{\varepsilon \to +0} \frac{u(1 \pm \varepsilon) - u(1)}{\pm \varepsilon} = w_0,
\]
so \( u \in C^1([0,\infty)) \), as desired.
\end{proof}

In the remainder of this section, we restrict attention to the case \( \beta \ge 1 \), \( w_0 \in (0,1) \). The solution of \eqref{sys_normalize} will henceforth be denoted by \( u(t; \beta, w_0) \).

From the first equation in \eqref{sys_normalize}, it is evident that \( u \) is non-decreasing. The next theorem shows that the derivative \( \frac{du}{dt}(t) \) never vanishes on an interval, implying that \( u(t) \) cannot remain constant at 1 over any nontrivial interval. It may approach 1 asymptotically but never reach a plateau.

\begin{Th}\label{Th:non-decreasing}
Let $w_0>0$ and \( a, b \in \R \) with \( a+1 < b \).  
Suppose \( u \in C^0([a,b]) \cap C^1([a+1,b]) \), and assume
\[
u(t) < 1 \quad \text{for all } t \in [a,a+1],
\]
and
\begin{align}\label{eq:delay}
\frac{du}{dt}(t) = w_0 \left|1 - u(t - 1)\right|^\beta \quad (a+1 \le t \le b).
\end{align}
If there exists \( t_* \in (a+1,b) \) such that \( u(t_*) = 1 \), then the following hold:
\begin{align}
\frac{du}{dt}(t_*) > 0, \quad 
u(t) \begin{cases}
< 1 & (a \le t < t_*), \\
= 1 & (t = t_*), \\
> 1 & (t_* < t \le b).
\end{cases}
\end{align}
\end{Th}

\begin{proof}
We first show \( u(t) < 1 \) for all \( t \in [a, t_*) \).  
Since \( u(t) < 1 \) for \( t \in [a, a+1] \), it suffices to consider \( t \in (a+1, t_*) \).  
Suppose for contradiction that there exists \( c \in (a+1, t_*) \) such that \( u(c) = 1 \).  
Since \eqref{eq:delay} implies \( \frac{du}{dt}(t) \ge 0 \) on \( [a+1, b] \), \( u \) is non-decreasing there, so \( u(t) = 1 \) for all \( t \in [c, t_*] \).

Then \( \frac{du}{dt}(t) = 0 \) on \( [c, t_*] \), and \eqref{eq:delay} yields \( u(t-1) = 1 \) for \( t \in [c, t_*] \), i.e., \( u(s) = 1 \) for \( s \in [c-1, t_*-1] \).  
If \( c-1>a+1 \), we repeat the same argument on the interval
\( [c-1,t_*-1] \), and by iterating this step (or using backward uniqueness for the DDE), we eventually obtain
\( u(s)=1 \) for some \( s\in[a,a+1] \), contradicting the assumption.
Hence \( u(t)<1 \) for all \( t\in[a,t_*) \).

Now for \( t \in [a+1, t_*+1) \cap [a+1, b] \), it follows that
\[
\frac{du}{dt}(t) = w_0 \left|1 - u(t - 1)\right|^\beta > 0,
\]
in particular \( \frac{du}{dt}(t_*) > 0 \).  
Integrating over any small interval after \( t_* \), we obtain \( u(t) > 1 \) for all \( t \in (t_*, b] \), completing the proof.
\end{proof}

\subsection{Numerical examples}\label{subsec:num}

Let $\dt \coloneqq 1/N$ be the time increment with $N \in \N$. Define $u_n = w_0 n\dt$ for $n = 0,1,\ldots, N$, and compute $u_n$ for $n = N+1, N+2, \ldots$ by iteratively solving the following equation:
\begin{align}\label{sys_discr}
    \frac{u_{n+1} - u_n}{\dt} = w_0|1 - u_{n-N}|^\beta.
\end{align}

\begin{center}
  \includegraphics[width=\columnwidth]{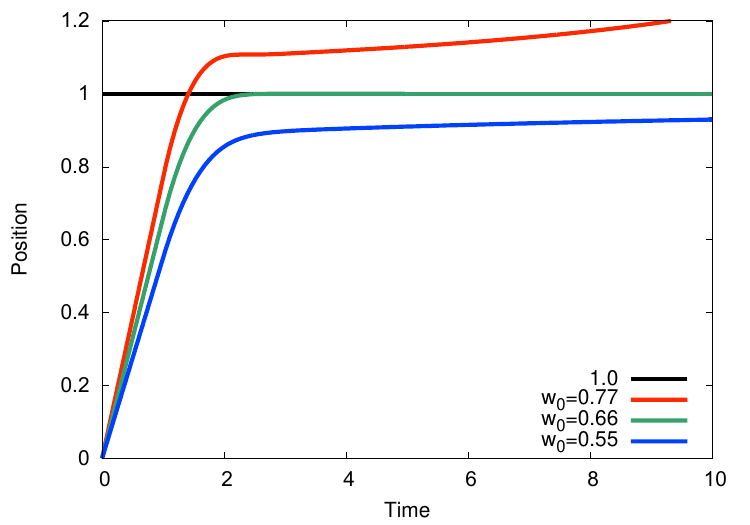}
  \captionof{figure}{Numerical results with $\beta = 2$ and $\dt = 0.01$ (red, green, and blue lines),
        compared to the constant value 1 (black line).
        The red, green, and blue lines show the evolution of $u(t)$ over $t \in [0,10]$ for
        $w_0 = 0.55$, $0.66$, and $0.77$, respectively.}
  \label{fig_ustar}
\end{center}

Here, we employ the explicit Euler scheme to numerically integrate the delay differential equation. 
Although Runge--Kutta methods are often used for improved numerical stability in delay differential equations, the Euler scheme produced sufficiently stable and reliable results in our simulations. 
Thus, we adopt the Euler scheme throughout this study.

Figure~\ref{fig_ustar} shows numerical results for $\beta = 2$, $w_0 = 0.55$, $0.66$, and $0.77$, with $\dt = 0.01$. 
These results indicate that a larger initial velocity $w_0$ leads to larger values of $u(t)$ over time.

Hence, for fixed $\beta \ge 1$, we expect that 
\begin{align}\label{cp1}
   &0 < w_0 < \tilde{w}_0 < 1 \notag\\
   \Rightarrow\,	&
   u(t;\beta,w_0) < u(t;\beta,\tilde{w}_0)
   \quad (t > 0),
\end{align}
and that there exists a threshold value $g(\beta) \in (0,1)$ such that
\begin{align}\label{gb}
    \left\{\begin{aligned}
    g(\beta) < w_0 < 1 
    &\Rightarrow 
    \begin{cases}
        ^\exists T > 1 \text{ s.t.} \\
        u(t;\beta,w_0) > 1 
        \text{ for } t > T,
    \end{cases}\\
    0 < w_0 \le g(\beta) 
    &\Rightarrow 
    \begin{cases}
        u(t;\beta,w_0) < 1 \text{ for } t > 0,\\
        \lim_{t \rightarrow \infty} u(t;\beta,w_0) = 1.
    \end{cases}
    \end{aligned}\right.
\end{align}
In other words, if $w_0 > g(\beta)$, then $u(t)$ overshoots the value 1. If $w_0 \le g(\beta)$, $u(t)$ remains strictly less than 1 and approaches 1 asymptotically as $t \to \infty$.

Using the bisection method, one can numerically estimate $g(\beta)$ for each $\beta$ (see Figure~\ref{fig_fit}).

\begin{center}
  \begin{minipage}{\linewidth}\centering
    \includegraphics[width=\linewidth]{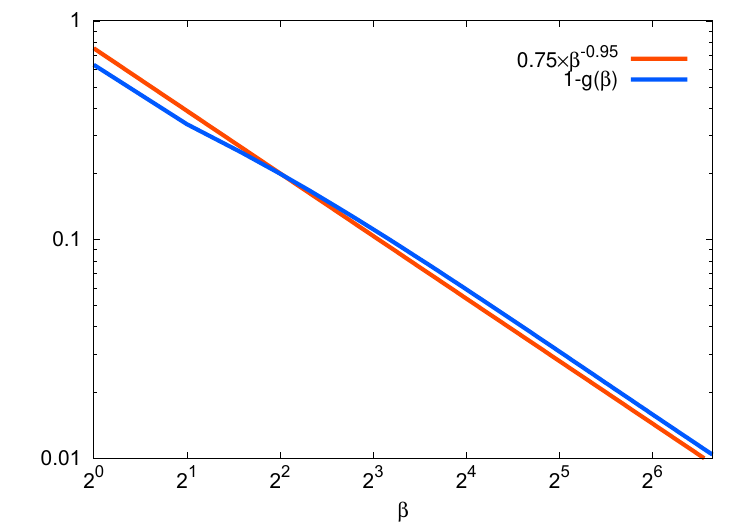}
    \par\small (a)
  \end{minipage}

  \vspace{0.5em}

  \begin{minipage}{\linewidth}\centering
    \includegraphics[width=\linewidth]{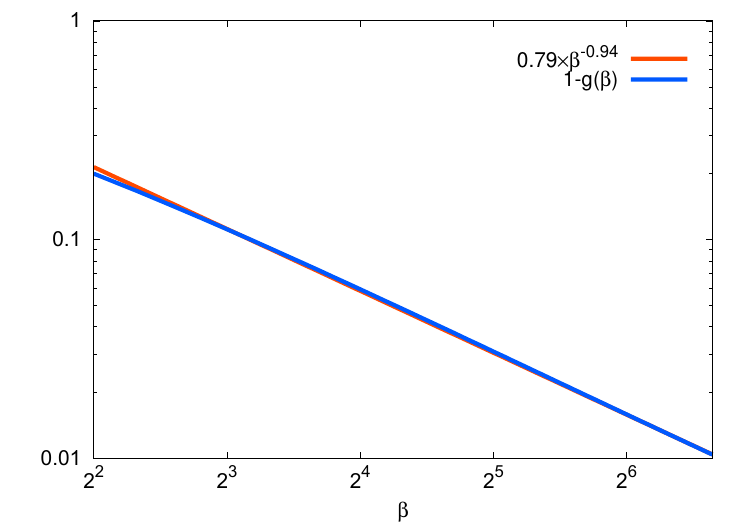}
    \par\small (b)
  \end{minipage}

  \vspace{0.5em}

  \begin{minipage}{\linewidth}\centering
    \includegraphics[width=\linewidth]{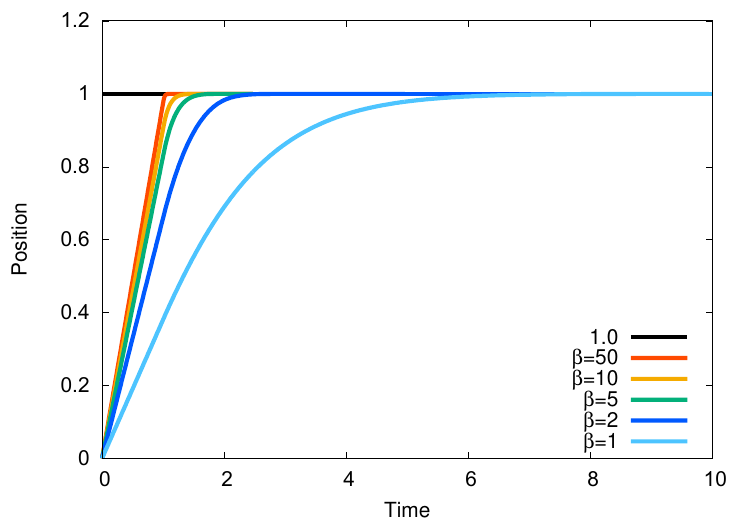}
    \par\small (c)
  \end{minipage}

  \captionof{figure}{Log-log plots of $1 - g(\beta)$ and its linear approximation over $\beta \in [1,100]$ (a) and $\beta \in [4,100]$ (b), and numerical simulations with $w_0 = g(\beta)$ for $\beta = 1,2,5,10,50$ (c).}
  \label{fig_fit}
\end{center}

Figure~\ref{fig_fit}(c) suggests the inequality
\begin{align}\label{beta_incre}
    \beta < \tilde{\beta} \Rightarrow 
    u(t, \beta, g(\beta)) < u(t, \tilde{\beta}, g(\tilde{\beta}))
    \quad (t > 0),
\end{align}
i.e., when the initial velocity $w_0 = g(\beta)$, the larger the value of $\beta$, the faster $u(t)$ approaches 1. 

Therefore, for fixed $\ell$, $\tau$, and $\beta$, the optimal initial velocity $v_0$ in \eqref{vel_sys} should be set as
\begin{align}\label{best_v}
    v_0 = \frac{\ell}{\tau} g(\beta),
\end{align}
which provides a solution to the press control problem discussed in Subsection~\ref{subsec:2.2}. It is also preferable to choose $\beta$ as large as possible.

Figures~\ref{fig_fit}(a) and (b) show that the relationship between $\beta$ and $1 - g(\beta)$ is nearly linear on a log-log scale. For $\beta \in [4,100]$, a linear regression yields the approximation
\[
    g(\beta) \approx 1 - 0.79 \times \beta^{-0.94}.
\]
Hence, a practical estimate of the appropriate initial velocity is
\[
    v_0 \approx \frac{\ell}{\tau} \left(1 - 0.79 \times \beta^{-0.94}\right).
\]
This approximation can be used as a convenient reference when choosing parameters for real-world applications.

However, if $\beta$ is large, the machine will decelerate rapidly near the target stopping position. As a result, the stopping accuracy may be sensitive to the time delay and measurement errors. Further discussion is provided in Section~\ref{sec:application}.

\subsection{Conjecture}\label{conjecture}

Based on the numerical experiments presented in Subsection~\ref{subsec:num}, we propose the following conjecture regarding the dynamics of the DDE \eqref{sys_normalize}.

\begin{Conj}\label{Conj:overshoot}
For $\beta \ge 1$, there exists a unique value $g(\beta) \in (0,1)$ such that the solution $u(t; \beta, w_0)$ of \eqref{sys_normalize} exhibits the following behavior:
\begin{enumerate}
\renewcommand{\labelenumi}{(\roman{enumi})}
\item 
    If $w_0 \in (g(\beta),1)$, there exists a constant $T_* = T_*(\beta, w_0) > 0$ such that
    \begin{align*}
        u(t; \beta, w_0)
        \begin{cases}
            < 1 & \text{for } 0 \le t < T_*,\\
            = 1 & \text{at } t = T_*,\\
            > 1 & \text{for } t > T_*.
        \end{cases}
    \end{align*}
\item 
    If $w_0 \in (0, g(\beta)]$, then
    \begin{align*}
        u(t; \beta, w_0)
        \begin{cases}
            < 1 & \text{for } t \ge 0,\\
            \to 1 & \text{as } t \to \infty.
        \end{cases}
    \end{align*}
\item The comparison properties \eqref{cp1} and \eqref{beta_incre} hold.
\end{enumerate}
\end{Conj}

Cases (i) and (ii) describe the two qualitatively distinct behaviors observed in our numerical simulations: overshooting and non-overshooting, respectively. The critical value $g(\beta)$ serves as a threshold separating these behaviors.

The monotonicity property \eqref{cp1} with respect to the initial value in (iii) enables the use of the bisection method to numerically identify $g(\beta)$.  
It is worth noting that this monotonicity has been proven analytically for the interval $0 < t \le 2$:

\begin{Prop}
    Fix $\beta \ge 1$. If
    \begin{align*}
        0 < w_0 < \tilde{w}_0 < 1,
    \end{align*}
    then
    \begin{align*}
        u(t; \beta, w_0) < u(t; \beta, \tilde{w}_0) \quad \text{for } 0 < t \le 2.
    \end{align*}
\end{Prop}

\begin{proof}
Since the case $0 < t \le 1$ follows directly from the initial condition, we focus on the interval $1 < t \le 2$.

Let $u(t) \coloneqq u(t; \beta, w_0)$. For $t \in [1,2]$, we compute
\begin{align*}
    u(t) &= u(1) + \int_1^t w_0 \left(1 - w_0(s-1)\right)^\beta \, ds \\
    &= w_0 + \frac{1}{\beta + 1} \left(1 - \left(1 - w_0(t - 1)\right)^{\beta + 1} \right),
\end{align*}
which is monotonically increasing with respect to $w_0$ for fixed $t \in [1,2]$, proving the result.
\end{proof}

This conjecture is also related to the mathematical analysis of finite-time blow-up phenomena in delay differential equations, which will be further discussed in Section~\ref{sec:blowup}.

\section{Applications}\label{sec:application}
\setcounter{equation}{0}

We consider the press control problem under a given maximum velocity constraint. In practical settings, press machines have a maximum velocity $\vmax$ due to safety and mechanical limitations. If $\beta$ is fixed, selecting the initial velocity as $\ell G/\tau~ (G\coloneqq g(\beta))$ is optimal according to (\ref{best_v}). Therefore, if $\ell G/\tau \le \vmax$, i.e., $\ell \le \ell_1\coloneqq \tau \vmax/G$, the press velocity is given by
\begin{align}
    V(X) = \frac{\ell G}{\tau}\left|\frac{\ell-X}{\ell}\right|^\beta.
\end{align}
Here, $X$ denotes the position information defined in \eqref{delay_pos} and is set to $0$ for $t<\tau$. If $\ell > \ell_1$, the press operates at the maximum velocity $\vmax$ initially, then switches modes just before an overshoot occurs:
\[
    V(X) = \left\{
        \begin{array}{ll}
            \vmax & \text{if } \ell_1 < \ell - X,\\
            {\DS \vmax\left|\frac{\ell-X}{\ell}\right|^\beta} & \text{if } \ell_1 \ge \ell - X.
        \end{array}
    \right.
\]
This can be summarized more compactly as:
\begin{align}\label{new_press}
    V(X) = v_0\left|\frac{\min\{\ell_1, \ell-X\}}{\min\{\ell_1, \ell\}}\right|^\beta,
\end{align}
where $v_0\coloneqq \min\{\ell G/\tau, \vmax\}$, $\ell_1\coloneqq \tau\vmax/G$, and $G\coloneqq g(\beta)$.

Let $\dt < \tau$, $N \coloneqq \lfloor\tau/\dt\rfloor$, and define $t_n = n\dt$ for $n \in \mathbb{N}$. For a numerical comparison between \eqref{old_press} and \eqref{new_press}, we formulate the following discrete problems.

\begin{Prob}\label{old_sys}
    Set $x_n = v_0 t_n$ for $n = 0,1,\ldots, N$, and compute $x_n$ for $n = N+1, N+2, \ldots$ by solving the recurrence
    \begin{align}\label{old_discr}
        \frac{x_{n+1} - x_n}{\dt} = v_0\left|\frac{\ell-x_{n-N}}{\ell}\right|.
    \end{align}
\end{Prob}

\begin{Prob}\label{new_sys}
    Set $x_n = v_0 t_n$ for $n = 0,1,\ldots, N$, and compute $x_n$ for $n = N+1, N+2, \ldots$ by solving the recurrence
    \begin{align}\label{new_discr}
        \frac{x_{n+1} - x_n}{\dt} 
        = v_0\left|\frac{\min\{\ell_1, \ell-x_{n-N}\}}{\min\{\ell_1, \ell\}}\right|^\beta.
    \end{align}
\end{Prob}
\medskip

Table \ref{tab_new} summarizes the time required to approach the target position in Problems \ref{old_sys} and \ref{new_sys} for $\beta=2,5,10,50$ (see also Figure~\ref{fig_new}). We set $\tau=40$, $\vmax=10$, and $\dt=0.5$. In Problem \ref{new_sys} with $\ell = 1000$ and $\beta = 50$, the computed trajectory $x_n$ exhibited a sharp deceleration just before reaching $\ell$, stalling approximately $0.2$ units short of the target within the simulated time.

For smaller values of $\beta$, the pressing time is longer, indicating that larger $\beta$ values are preferable for efficiency. However, when $\beta$ becomes too large, the arrival time plateaus, and the rapid variation of the velocity near $\ell$ increases the risk of missing the target stroke position.

\begin{table*}[htpb]\centering
    \begin{tabular}{c|c|c|c}
        $\ell$ & Press control model & $\min\{t_n \mid \ell - x_{n-N} < 1\}$ & $\min\{t_n \mid \ell - x_{n-N} < 0.1\}$ \\ \hline
        100 & Problem \ref{old_sys} & 235 ($235/235=1$) & 274 ($274/274=1$) \\
         & Problem \ref{new_sys} with $\beta=2$ & 125 ($125/235\simeq0.53$) & 142 ($142/274\simeq0.52$) \\
         & Problem \ref{new_sys} with $\beta=5$ & 98.5 ($98.5/235\simeq0.42$) & 108 ($108/274\simeq0.39$) \\
         & Problem \ref{new_sys} with $\beta=10$ & 88.5 ($88.5/235\simeq0.38$) & 95.5 ($95.5/274\simeq0.35$) \\
         & Problem \ref{new_sys} with $\beta=50$ & 81.0 ($81.0/235\simeq0.34$) & 83.0 ($83.0/274\simeq0.30$) \\ \hline 
        1000 & Problem \ref{old_sys} & 274 ($274/274=1$) & 285 ($285/285=1$)\\
         & Problem \ref{new_sys} with $\beta=2$ & 178 ($178/274\simeq0.65$) & 191 ($191/285\simeq0.67$)\\
         & Problem \ref{new_sys} with $\beta=5$ & 158 ($158/274\simeq0.58$) & 166 ($166/285\simeq0.58$)\\
         & Problem \ref{new_sys} with $\beta=10$ & 149 ($149/274\simeq0.54$) & 158 ($158/285\simeq0.55$)\\
         & Problem \ref{new_sys} with $\beta=50$ & 141 ($141/274\simeq0.51$) & --- \\ \hline 
    \end{tabular}
    \caption{Press times $\min\{t_n \mid \ell - x_{n-N} < 1\}$ and $\min\{t_n \mid \ell - x_{n-N} < 0.1\}$ for Problems \ref{old_sys} and \ref{new_sys} with $\beta=2,5,10,50$, $\tau=40$, $\vmax=10$, and $\dt=0.5$. For Problem \ref{new_sys} with $\beta=50$ and $\ell=1000$, the trajectory stalled approximately $0.2$ units short of the target within the simulation period.}
    \label{tab_new}
\end{table*}

\begin{center}
  \includegraphics[width=\columnwidth,clip]{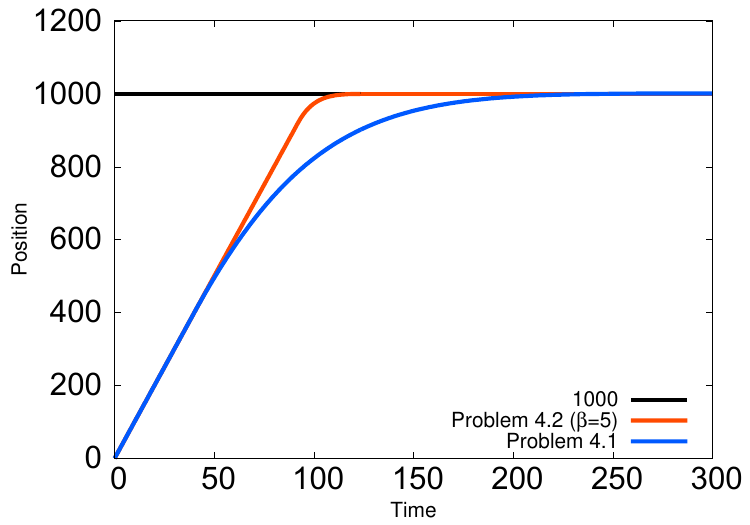}
  \captionof{figure}{Numerical results for Problem \ref{old_sys} (blue line) and Problem \ref{new_sys} (red line) with $\ell=1000$, $\vmax=10$, $\tau=40$, $\beta=5$, and $\dt=0.5$.}
  \label{fig_new}
\end{center}

\begin{Rem}
    {\rm
    It takes $\max\{\tau, \ell/\vmax\}$ time to reach the target press position. Moreover, due to the time delay, an additional $\tau$ is required to determine when to stop. Thus, from the start of pressing to complete stoppage, the total time is $\max\{\tau, \ell/\vmax\}+\tau$. In the examples shown in Table \ref{tab_new} and Figure~\ref{fig_new}, when $\ell=100$, it takes over $80$ time units ($40 + 40$), and when $\ell=1000$, it takes over $140$ time units ($1000/10 + 40$) to reach the target position.
    }
\end{Rem}

\section{Blow-up Phenomenon of DDEs}\label{sec:blowup}
\setcounter{equation}{0}

As mentioned in the introduction, the behavior of DDE solutions is complex, and many aspects remain unclear. For example, regarding the finite-time blow-up phenomenon of solutions, \cite{EJ06,IN22} analyze DDEs of the type
\[
    \frac{dz}{dt}(t) = f(z(t))h(z(t-1)).
\]
In this section, we explain how the mathematical analysis and overshoot conjecture for DDEs in Section~\ref{sec:num} can be applied to the blow-up problem for solutions of DDEs of the above type.

For the nondimensionalized system \eqref{sys_normalize}, applying the variable transformation $z = \frac{1}{1-u}$ and noting that $\frac{du}{dt} = z^{-2}\frac{dz}{dt}$, we obtain
\begin{equation}\label{sys_inverse}\left\{\begin{aligned}
    \frac{dz}{dt}(t)
    &= w_0\frac{(z(t))^2}{(z(t-1))^\beta} 
    && \text{if } t>1,\\
    z(t) &= \frac{1}{1-w_0 t} && \text{if } 0 \le t \le 1,
\end{aligned}\right.\end{equation}
In this formulation, whether or not $u$ overshoots (i.e., exceeds 1) corresponds to whether or not $z$ blows up in finite time.

From Theorem~\ref{Th:non-decreasing}, we can deduce that if blow-up occurs, the blow-up rate satisfies \( z(t) \sim \frac{c}{t_*-t} \) as \( t \to t_*^- \), which is of the same order as the blow-up rate of the ordinary differential equation \( z'(t) = \alpha z(t)^2 \).

\begin{Th}\label{th:blowup}
Suppose that $h\in C^0((0,\infty))$ with $h(s)>0$ for $s>0$.
Let $a, t_* \in \mathbb{R}$ with $a + 1 < t_*$. 
Assume that 
$z \in C^0([a, t_*)) \cap C^1([a+1, t_*))$, 
$z(t) > 0$, and 
\begin{align*}
    \frac{dz}{dt}(t)
    = (z(t))^2h(z(t-1)) \quad (a+1 \le t < t_*).
\end{align*}
If $z(t) \to \infty$ as $t \to t_*^-$, then
\begin{align}\label{lim:finite}
    z(t) = \frac{c}{t_* - t} + o\left(\frac{1}{t_* - t}\right) 
    \quad \text{as } t \to t_*^-,
\end{align}
where $c \coloneqq h(z(t_* - 1))^{-1}  >0$, and $o$ is Landau's symbol.
\end{Th}

\begin{proof}
Define
\begin{align*}
    y(t) \coloneqq 
    \begin{cases}
        {\DS - \frac{1}{z(t)}} & (a \le t < t_*)\\
        0 & (t = t_*).
    \end{cases}
\end{align*}
Then $y \in C^0([a, t_*])$ and $y(t_*) = 0$.

For $a + 1 \le t < t_*$, we compute
\begin{align*}
    \frac{dy}{dt}(t)
    &= \frac{1}{z(t)^2} \cdot \frac{dz}{dt}(t)\\
    &= \frac{1}{z(t)^2}\cdot (z(t))^2h(z(t-1))\\
    &=h(z(t-1)).
\end{align*}
In particular, we obtain
\begin{align*}
    \lim_{t \to t_*^-} \frac{dy}{dt}(t) = h(z(t_* - 1)) = c^{-1} > 0.
\end{align*}
Since $y$ is continuous on $[a,t_*]$, differentiable on $[a+1,t_*)$, and the limit of $\frac{dy}{dt}(t)$ as $t\to t_*^-$ exists and is finite, the mean value theorem implies that $y$ is left-differentiable at $t_*$ with
\[
	\frac{dy}{dt}(t_*) = \lim_{t\to t_*^-} \frac{dy}{dt}(t) = c^{-1}.
\]
In particular, $y \in C^1([a+1, t_*])$.
Hence, we compute
\begin{align*}
    \lim_{t \to t_*^-} (t_* - t) z(t)
    &= \lim_{t \to t_*^-} \frac{t_* - t}{ - y(t)}\\
    &= \lim_{t \to t_*^-} \frac{t_* - t}{y(t_*) - y(t)}\\
    &= \left( \frac{dy}{dt}(t_*) \right)^{-1} = c,
\end{align*}
which implies the asymptotic behavior \eqref{lim:finite}.
\end{proof}
\begin{Rem}{\rm
In Theorem~\ref{th:blowup}, if $h(s) = w_0 s^{-\beta}$, then for $t \in (t_* - \varepsilon, t_*)$ (with $\varepsilon > 0$ sufficiently small), the function $u(t) = 1 + y(t)$ satisfies $u(t) \in (0,1)$ and
\begin{align*}
\frac{du}{dt}(t) &= \frac{dy}{dt}(t) 
= h(z(t-1)) \\
&= h\left( \frac{1}{1 - u(t-1)} \right)
= w_0 (1 - u(t-1))^\beta,
\end{align*}
which is exactly the same equation as \eqref{eq:delay}.  
This shows that Theorem~\ref{th:blowup} is simply an alternative formulation of Theorem~\ref{Th:non-decreasing}.
}
\end{Rem}

Furthermore, Conjecture \ref{Conj:overshoot} suggests a complete characterization of blow-up behavior in terms of the initial parameter $w_0$. Specifically, it predicts the existence of a critical threshold $g(\beta)$ that determines whether finite-time blow-up occurs. Using the transformation $z = \frac{1}{1-u}$, the behavior described in Conjecture \ref{Conj:overshoot} can be directly translated into statements about the blow-up properties and comparison principles for solutions of \eqref{sys_inverse}.

\begin{Prop}\label{Prop:blowup}
Assume that Conjecture \ref{Conj:overshoot} holds. 
Let $g(\beta)$ and $T_*=T_*(\beta, w_0)>0$ be as described in Conjecture \ref{Conj:overshoot}.
For $\beta\ge 1$, the solution $z(t;\beta,w_0)$ of \eqref{sys_inverse} satisfies the following properties.
\begin{enumerate}
\renewcommand{\labelenumi}{(\roman{enumi})}
\item 
    For $w_0 \in (g(\beta),1)$, it holds that
    \begin{align*}
        z(t;\beta, w_0)
        \begin{cases}
            <\infty&(0\le t <T_*)\\
            \to\infty&(t \to T_*^-).
        \end{cases}
    \end{align*}
\item 
    For $w_0\in (0,g(\beta)]$, it holds that
    \begin{align*}
        z(t;\beta, w_0)
        \begin{cases}
            <\infty&(t\ge 0)\\
            \to \infty&(\text{as } t\to \infty ).
        \end{cases}
    \end{align*}
\item 
    The comparison properties 
    \begin{align*}
        &w_0 < \tilde{w}_0 \\ \Rightarrow\,	&
        z(t;\beta,w_0) < z(t;\beta,\tilde{w}_0)\quad(t>0)
    \end{align*}
    and 
    \begin{align*}
        &\beta < \tilde{\beta} \\ \Rightarrow\,	&
        z(t; \beta, g(\beta)) < z(t; \tilde{\beta}, g(\tilde{\beta}))
        \quad(t>0)
    \end{align*}
    hold.
\end{enumerate}
\end{Prop}

\section{From Industrial Challenge to Mathematical Development}\label{sec:conclusion}
\setcounter{equation}{0}
In this study, we investigated the press control problem in straightening machines—used to reduce post-heat-treatment warping and improve the linear accuracy of automotive shafts—from a mathematical viewpoint. Among various factors affecting the control performance, we focused on the small but non-negligible time delay caused by communication and information processing between the machine and the control system. This analysis was grounded in empirical data from actual industrial equipment. To address this, we proposed a novel control strategy that generalizes conventional linear velocity control to a polynomial $\beta$-control of arbitrary degree $\beta \ge 1$.

The proposed control strategy led to a mathematical formulation as a delay differential equation (DDE) problem, with the initial velocity serving as a control parameter. In Subsection~\ref{subsec:scale}, we introduced a nondimensionalized version of the model and rigorously established several fundamental mathematical properties: the global existence of a unique solution, and the characterization of overshoot behavior. These results provide a solid theoretical foundation for analyzing the control problem.

In Subsection~\ref{subsec:num}, numerical simulations revealed the existence of a threshold initial velocity $w_0 = g(\beta)$ that determines whether overshoot occurs, depending on the value of $\beta$. These observations were formulated as a mathematical conjecture in Subsection~\ref{conjecture}, which also included partial evidence supporting its validity.

In Section~\ref{sec:application}, we proposed a new algorithm for the press control problem under a maximum velocity constraint, utilizing this conjecture, and demonstrated its effectiveness through numerical examples. Furthermore, in Section~\ref{sec:blowup}, we showed that our analysis and  the conjecture in Section~\ref{sec:num} are closely related to the finite-time blow-up phenomenon in the solutions of DDEs, derived an estimate of the blow-up rate, and showed the existence of a threshold value that would hold if the conjecture is true.

As illustrated in Fig.~\ref{fig_realdelay}, the time delay in the actual press-machine data is not constant; 
rather, it exhibits a time-varying structure $\tau(t)$. 
In this paper, as a first step toward a full analysis, we restrict our study to the case of a constant delay and develop the theoretical framework under this assumption. 
Nevertheless, the results in Subsection~\ref{subsec:scale} can be extended to the case of a time-varying delay $\tau(t)$, based on the general theory of functional differential equations (see, e.g., \cite{HVL93}). 
A thorough investigation of such extensions is left for future research.

In particular, understanding how a time-varying delay affects the stability and performance of the system—especially when fluctuations around the mean delay become large—is an important and highly intriguing problem. 
It would be natural to pursue this direction within a stochastic framework, for example by formulating the model as a delay differential equation with stochastically varying delay. 
We believe that the present work provides an essential first step toward such promising future developments.

This study not only addressed a real-world industrial challenge through mathematical modeling and control design, but also contributed new theoretical insights into the dynamics of delay differential equations.
It is widely recognized that the collaboration between industry and mathematics will become increasingly important in the future. Though modest, we believe this study represents a meaningful first step in contributing to such interdisciplinary efforts.

As for future mathematical directions, a rigorous proof of Conjecture~\ref{Conj:overshoot} remains an important open problem. 
Although this conjecture lies at the core of the present study, its proof does not appear straightforward, and at this stage we were unable to obtain a fully rigorous analysis. 
Nevertheless, careful numerical experiments strongly support its validity. 
We therefore present it explicitly as a conjecture and leave its mathematical resolution for future work.

Additionally, the relationship between control and blow-up behavior in DDEs (Section~\ref{sec:blowup}) is still an underexplored area, and further research is highly anticipated.

\section*{Acknowledgement}
This research originated from discussions held at the Study Group for Industrial Problems, hosted by the Graduate School of Mathematical Sciences, the University of Tokyo. The workshop was held over several one-week sessions between February 2014 and February 2024. The authors would like to thank the participants for their valuable contributions during these sessions.

The industrial context and motivation for this research were provided through collaborative discussions with engineers from Towa Seiki Co. Ltd. The data and technical insights obtained through this collaboration played a crucial role in shaping the direction of the study.

This work was partially supported by Towa Seiki Co., Ltd.
M.Y. was supported by JSPS KAKENHI Grant Numbers JP20H00117 and JP21K18142.
Y.L. was supported by JSPS KAKENHI Grant Numbers JP22K13954, JP23KK0049 and Guangdong Basic and Applied Basic Research Foundation (No. 2025A1515012248).

\end{multicols}
\end{document}